\documentclass[12pt, a4paper, twosided, reqno]{amsart}
\usepackage{enumerate, cite}
\usepackage{hyperref}

\newtheorem{theorem}{Theorem}
\newtheorem{lemma}{Lemma}

\newtheorem{example}{Example}

\newtheorem*{thA}{Theorem A}
\newtheorem*{thB}{Theorem B}
\newtheorem*{thC}{Theorem C}

\allowdisplaybreaks
\author[N. Gahlian and G. Pant]{Nidhi Gahlian and Garima Pant}
\address{nidhi gahlian; department of mathematics, university of delhi, delhi-110007, india.}
\email{nidhigahlyan81@gmail.com}
\address{garima pant; department of mathematics, university of delhi, delhi-110007, india.}
\email{garimapant.m@gmail.com}

\thanks {Research work of the first author is supported by research fellowship from Department of Science and Technology(INSPIRE), New Delhi, India.}
\thanks {Research work of the second author is supported by research fellowship from University Grants Commission (UGC), New Delhi, India.}
\title[On Hayman conjecture]{On Hayman Conjecture for Paired Complex Delay-Differential Polynomials}
\begin{document}
	\maketitle
	
	\begin{abstract}
	We study Hayman conjecture for different paired complex polynomials under certain conditions. The zeros distribution of $f^{n}(z)L(g)-a(z)$ and $g^{n}(z)L(f)-a(z)$ was studied by Gao and Liu \cite{gaoliu} for $n\geq 3$. In this paper, we work on the zeros distribution of $f^{2}(z)L(g)-a(z)$ and $g^{2}(z)L(f)-a(z)$, where $a(z)$ is a non-zero small function of both $f(z)$ and $g(z)$, and  $L(h)$ takes the $k$th derivative $h^{(k)}(z)$ or shift $h(z+c)$ or difference $h(z+c)-h(z)$ or delay-difference $h^{(k)}(z+c)$, here $k\geq 1$ and $c$ is a non-zero constant. Moreover, we discuss Hayman conjecture for paired complex differential polynomials when $n=1.$
	\end{abstract}

		               \section{\textbf{Introduction}}

To understand this paper, we must know the basic facts of Nevanlinna's value distribution theory. For a meromorphic function $f$, $n(r,f)$ $N(r,f), m(r,f)$ and  $T(r,f)$ denote un-integrated counting function, integrated counting function, proximity function and characteristic function respectively. We also use first main theorem and second main theorem of Nevanlinna for a meromorphic function $f$, see  \cite{yanglo,hayman,ilaine}. We present elementary definitions of order of growth $\rho(f)$, hyper-order of growth $\rho_{2}(f)$ and exponent of convergence of zeros for a meromorphic function $f$ to make the paper self contained. 
$$\rho(f)=\limsup_{r\to\infty}\frac{\log T(r,f)}{\log r},$$ 
$$\rho_{2}(f)=\limsup_{r\to\infty}\frac{\log \log T(r,f)}{\log r}$$ and 
$$\lambda(f)=\limsup_{r\to\infty}\frac{\log n(r,1/f)}{\log r}=\limsup_{r\to\infty}\frac{\log N(r,1/f)}{\log r}.$$   \
A meromorphic function $g(z)$ is a small function of $f(z)$ if $T(r,g)=S(r,f)$ and converse is also true, here $S(r,f)$ denotes such quantities which are of growth $o(T(r,f))$ as $r\to\infty$, outside of a possible exceptional set of finite linear measure.\\ 
In 2022, Gao and Liu \cite{gaoliu} considered the Hayman conjecture of paired differential polynomials and paired delay-differential polynomials. In particular, they considered the zeros distribution of $f^{n}(z)L(g)-a(z)$ and $g^{n}(z)L(f)-a(z)$, where $n\in\mathbb{N}$, $a(z)$ is a non-zero small function of $f(z)$ and $g(z)$. Also, $L(h)$ holds one of the following conditions:
\begin{enumerate}
\item $L(h)=h^{(k)}(z),k\geq 1$
\item $L(h)=h(z+c)$, $c\in \mathbb{C}\setminus{0}$
\item $L(h)=h(z+c)-h(z)$, $c\in \mathbb{C}\setminus{0}$
\item $L(h)=h^{(k)}(z+c),k\geq 1$ and  $c\in \mathbb{C}\setminus{0}$.
\end{enumerate}	
Let $M$ denotes the function class of all transcendental meromorphic functions, $M'$ denotes function class of transcendental meromorphic functions of hyper-order less than $1$, $E$ denotes function class of all transcendental entire functions and $M'$ denotes function class of transcendental entire functions of hyper-order less than $1$. Then Gao and Liu \cite{gaoliu} proved the following result:
\begin{theorem}\rm\label{gaoliuth}
If $L(h)$ satisfies any one of the following conditions:
\begin{enumerate}
\item $L(h)=h^{(k)}(z), n\geq k+4$ and $h\in M$ or $n\geq 3$ and $h\in E$;
\item $L(h)=h(z+c), n\geq 4$ and $h\in M'$ or $n\geq 3$ and $h\in E'$;
\item $L(h)=h(z+c)-h(z), n\geq 5$ and $h\in M'$ or $n\geq 3$ and $h\in E'$;
\item $L(h)=h^{(k)}(z+c), n\geq k+4$ and $h\in M'$ or $n\geq 3$ and $h\in E'$.
\end{enumerate}
Then at least one of  $f^{n}(z)L(g)-a(z)$ and $g^{n}(z)L(f)-a(z)$ have infinitely many zeros.
\end{theorem}
In the same paper they raised a question namely `Question $1$': Can we reduce $n\geq 3$ to $n\geq 2$ for entire functions $f$ and  $g$ in $E$ or $E'$?  And what is the sharp value of $n$ for meromorphic functions $f$ and  $g$ in $M$ or $M'?$ We give partial answer to this question.\\
Let $M^*$ denotes the function class of transcendental meromorphic function $f$ such that $N(r,f)+N(r,1/f)=S(r,f)$ and $M^{**}$ denotes the function class of transcendental meromorphic function $f$ such that  $N(r,f)+N(r,1/f)=S(r,f)$ with $\rho_{2}(f)<1$. Then the following results hold:

\begin{thA}\rm\label{mainth1}
Suppose that $f$ and $g$ are functions from $M^*$ class, then at least one of $f^{2}(z)g^{(k)}(z)-a(z)$ and $g^{2}(z)f^{(k)}(z)-a(z)$ have infinitely many zeros, where $k\geq 1$, $a(z)$ is a small function of $f(z)$ and $g(z)$.
\end{thA}
 
\begin{example}\rm
Let $f(z)=e^{z}$, $g(z)=e^{-z}$ and $a(z)$ be any non-zero polynomial, then $f^{2}(z)g^{'}(z)-a(z)$ and $g^{2}(z)f^{'}(z)-a(z)$ both have infinitely many zeros.  
\end{example}

 \begin{thB}\rm\label{mainth2}
 Suppose that $f$ and $g$ are functions from $M^{**}$ class, $a(z)$ is a small function of $f(z)$ and $g(z)$, $c$ is a non-zero complex number and $L(h)$ satisfies any one of the following conditions:
 \begin{enumerate}
 \item $L(h)=h(z+c)$;
 \item $L(h)=h(z+c)-h(z)$;
 \item $L(h)=h^{(k)}(z+c), k\geq 1$.
 \end{enumerate}
Then at least one of $f^{2}(z)L(g)-a(z)$ and $g^{2}(z)L(f)-a(z)$ have infinitely many zeros. 
 \end{thB}

\begin{example}\rm
Let $f(z)=e^{z}$, $g(z)=e^{-z}$, $a(z)$ be any non-zero polynomial and $c$ is a non-zero number, then
\begin{enumerate}[(i)]
\item $f^{2}(z)g(z+c)-a(z)$ and $g^{2}(z)f(z+c)-a(z)$ both have infinitely many zeros.
\item $f^{2}(z)(g(z+c)-g(z))-a(z)$ and $g^{2}(z)(f(z+c)-f(z))-a(z)$ both have infinitely many zeros.
\item  $f^{2}(z)g^{(k)}(z+c)-a(z)$ and $g^{2}(z)f^{(k)}(z+c)-a(z)$ both have infinitely many zeros.
\end{enumerate} 
\end{example}

It was observed that Theorem \ref{gaoliuth} is not true for $n=1$, see \cite [Remark 1.2]{gaoliu}. As we can see that if $f(z)=ze^{4z}$ and $g(z)=z^{2}e^{-4z}$, then $f(z)g^{'}(z)-a(z)$ and $g(z)f^{'}(z)-a(z)$ both have finitely many zeros provided $a(z)$ is any polynomial in $z$ except $2z^{2}-4z^{3}$ and $z^{2}+4z^{3}$. Next we are going to study the $(1)$ part of Theorem \ref{gaoliuth} for the case when $n=1$ and $k=1$. The following result shows that under the certain conditions on $f$ and $g$, we get the same conclusion as in Theorem \ref{gaoliuth} for $n=1$. Before stating the result, let 
$F^{*}=\{f=\alpha(z)e^{P(z)}-k: \alpha(z)$ and $P(z)$ are non-zero polynomial and non-constant polynomial respectively, and $k\in \mathbb{C}\setminus\{0\}\}$.

\begin{thC}\rm\label{mainth3}
Let $f$ and $g$ be the functions from $F^{*}$ class, then at least one of $fg^{'}-a(z)$ and $gf^{'}-a(z)$ have infinitely many zeros, where $a(z)$ is a non-zero polynomial.
\end{thC}

\begin{example}\rm
Let $f(z)=e^{2z}+3$, $g(z)=e^{-2z}+4$ and $a(z)$ be any non-zero polynomial then at least one of $fg^{'}-a(z)$ and $gf^{'}-a(z)$ have infinitely many zeros.
\end{example}

Prior to Theorem \ref{gaoliuth}, Hayman \cite{hayman1} proved some results regarding zero distribution of complex differential polynomial in $1959$. He also presented a conjecture which is as follows: 	
\begin{theorem}\rm\label{haymanconjecture}
Suppose that $f$ is a transcendental meromorphic function then $f^{n}(z)f^{'}(z)-c$ has infinitely many zeros, where $n$ is a positive integer and $c$ is any non-zero complex number.
\end{theorem}

The above conjecture has been proved completely, see \cite{hayman1,mues,be}. After Hayman, Laine and
Yang \cite[Theorem 2]{laineyang} studied the zero distribution of complex difference polynomials and proved the following result.

\begin{theorem}\rm
Suppose that $f(z)$ is a transcendental entire function of finite order and $c$ is a non-zero constant, then  then $f^{n}(z)f(z+c)-a$ has infinitely many zeros, where $n\geq 2$ and $a$ is a non-zero constant.
\end{theorem}
Later on, many authors have made improvements on the above theorem such as
$a$ is replaced by any non-zero polynomial, $f(z+c)$ and $a$ are replaced by  $f(z+c)-f(z)$ and any non-zero polynomial respectively, $f^{n}(z)$ and $a$ are replaced by polynomial of degree $n$ and non-zero function of growth $o(T(r,f))$ respectively, see \cite{liuyang, luolin}.\\
In 2011, Liu et al. \cite{liuliucao} also proved the above theorem for the case when $f(z)$ is a transcendental meromorphic function with $\rho_{2}(f)<1$ and $n\geq 6$.
They also gave counter examples for $n\leq 3$. After Liu et al., the same result was proved for $n\geq 4$ by Wang and Ye \cite{wangye}.\\
In $2014$, Liu et al.\cite{liuliuzhou} studied delay-differential version of Theorem \ref{haymanconjecture}. Then many improvements have been made by other authors, see \cite{wangye, lainelatreuch, liulaineyang}.

                     \section{\textbf{Auxiliary results}}

In this section we present some known basic facts and results which we will use in the next section. For a set $I\subset (0,\infty)$, the linear measure is defined by $m(I)=\int_{I}\,dt$ and for a set $J\subset (1,\infty)$, the logarithmic measure is defined by $m_{l}(J)=\int_{J}\frac{1}{t}\,dt$.

\begin{lemma}\rm\cite{lly}\label{llylem}
Suppose that $T:[0,\infty)\to [0,\infty)$ is a non-decreasing continuous function having $\rho_{2}(T)<1$ and $c$ is a non-zero real number. If $\delta\in(0,1-\rho_{2}(T))$, then
$$T(r+c)=T(r)+o\left(\frac{T(r)}{r^{\delta}}\right).$$
\end{lemma}

The following lemma gives estimate of counting function corresponding to the zeros of derivative of non-constant meromorphic function $f$.
\begin{lemma}\rm\label{yangyilem}\cite{gaoliu}
	Suppose that $f$ is a non-constant meromorphic function, then
	$$N\left(r,\frac{1}{f^{(k)}(z)}\right)\leq N\left(r,\frac{1}{f(z)}\right)+k\bar{N}(r,f(z))+S(r,f(z)), $$
	where $k\geq 1$.
\end{lemma}

The following two lemmas can be easily obtained by applying basic facts of Nevanlinna theory.
\begin{lemma}\rm\label{implem1}
	If $f$ is a function from $M^*$ class, then 
	$N(r,f^{(k)})=S(r,f)$ and $N\left(r,\frac{1}{f^{(k)}}\right)=S(r,f)$.
\end{lemma}
\begin{proof}
Given that $f$ is a transcendental meromorphic function satisfying
	\begin{equation}\label{giveneq}
		N(r,f)+N(r,1/f)=S(r,f).
	\end{equation}
	We know that if $f$ has a pole of order $m$ at $z_{0}$, then $f^{(k)}$ has a pole of order $m+k\leq (k+1)m$ at $z_{0}$. Thus
	\begin{equation*}
		N(r,f^{(k)})\leq (k+1)N(r,f).
	\end{equation*}
	Using equation \eqref{giveneq}, we have $N(r,f^{(k)})=S(r,f)$.\\
	Next, using equation \eqref{giveneq} to Lemma \rm\ref{yangyilem}, we have $N\left(r,\frac{1}{f^{(k)}}\right)=S(r,f)$.   
\end{proof}

\begin{lemma}\rm\label{llyimplem}
If $f$ is a function from $M^{**}$ class, $c$ is a non-zero complex number and $k\geq 1$, then
\begin{enumerate}
\item $N(r,f(z+c))=S(r,f)$ and  $N\left(r,\frac{1}{f(z+c)}\right)=S(r,f).$
\item $N(r,f(z+c)-f(z))=S(r,f)$ and $N\left(r,\frac{1}{f(z+c)-f(z)}\right)=S(r,f).$
\item $N(r,f^{(k)}(z+c))=S(r,f)$ and $N\left(r,\frac{1}{f^{(k)}(z+c)}\right)=S(r,f).$
\end{enumerate} 
\end{lemma}
\begin{proof}
\begin{enumerate}
\item By the simple observation and using Lemma \ref{llylem}, we have\\

$N(r,f(z+c))\leq N(r+|c|,f)\leq N(r,f)+S(r,f)=S(r,f)$ and
$$N\left(r,\frac{1}{f(z+c)}\right)\leq N\left(r+|c|,\frac{1}{f}\right)\leq N\left(r,\frac{1}{f}\right)+S(r,f)=S(r,f).$$
\item $N(r,f(z+c)-f(z))=S(r,f)$ is obvious by $(1)$ part and 
\begin{align*}
N\left(r,\frac{1}{f(z+c)-f(z)}\right)&=N\left(r,\frac{1}{f(z)}.\frac{1}{\frac{f(z+c)-f(z)}{f(z)}}\right)\\
&\leq N\left(r,\frac{1}{f(z)}\right)+N\left(r,\frac{1}{\frac{f(z+c)}{f(z)}-1}\right)\\
&\leq T\left(r,\frac{f(z+c)}{f(z)}\right)+S(r,f)\\
&=m\left(r,\frac{f(z+c)}{f(z)}\right)+N\left(r,\frac{f(z+c)}{f(z)}\right)+S(r,f)\\
&=S(r,f).
\end{align*} 
\item As we know that $N(r,f^{(k)}(z+c))\leq (k+1)N(r,f(z+c))$,
 then applying $(1)$ part, we get $N(r,f^{(k)}(z+c))=S(r,f)$.
Also applying $(1)$ part into Lemma \ref{yangyilem}, we get $N\left(r,\frac{1}{f^{(k)}(z+c)}\right)=S(r,f).$
\end{enumerate}
\end{proof}

\begin{lemma}\rm\label{llylemma}\cite{lly}
Suppose that $f$ is a transcendental meromorphic function such that $\rho_{2}(f)<1$ and $k\geq 0$, then
$$m\left(r,\frac{f^{(k)}(z+c)}{f(z)}\right)=S(r,f),$$
outside of a possible exceptional set of finite logarithmic measure.
\end{lemma}

The following two lemmas are obtained by using some basic facts of the Nevanlinna theory.
\begin{lemma}\label{implem2}
If $f$ is a function from $M^*$ class, then $T\left(r,\frac{1}{f^{(k)}}\right)\leq T(r,f)+S(r,f)$.
\end{lemma}
\begin{proof}
Given that $f$ is a transcendental meromorphic function satisfying
\begin{equation}\label{giveneqq}
	N(r,f)+N(r,1/f)=S(r,f).
\end{equation}
Next
\begin{align*}
T\left(r,\frac{1}{f^{(k)}}\right)&=T\left(r,\frac{1}{f^{(k)}/f}.\frac{1}{f}\right) \\
&\leq T\left(r,\frac{1}{f^{(k)}/f}\right)+T\left( r,\frac{1}{f}\right)+O(1).
\end{align*}
Using first main theorem of Nevanlinna, we have
\begin{align*}
T\left(r,\frac{1}{f^{(k)}}\right) &\leq T\left( r,\frac{f^{(k)}}{f}\right) +T(r,f)+O(1)\\
&\leq m\left( r,\frac{f^{(k)}}{f}\right) +N\left( r,\frac{f^{(k)}}{f}\right) +T(r,f)+O(1)\\
&\leq S(r,f)+N(r,f^{(k)})+N\left( r,\frac{1}{f}\right) +T(r,f).
\end{align*}
Applying Lemma \ref{implem1}, we have
$$T\left( r,\frac{1}{f^{(k)}}\right) \leq T(r,f)+S(r,f).$$
\end{proof}

\begin{lemma}\rm\label{implem3}
If $f$ is a function from $M^{**}$ class, then
\begin{enumerate}
\item $T\left(r,\frac{1}{f(z+c)}\right)\leq T(r,f)+S(r,f).$
\item $T\left(r,\frac{1}{f(z+c)-f(z)}\right)\leq T(r,f)+S(r,f).$
\item  $T\left(r,\frac{1}{f^{(k)}(z+c)}\right)\leq T(r,f)+S(r,f).$
\end{enumerate}
\end{lemma}

\begin{proof}
\begin{enumerate}
\item This is easily obtained by applying first fundamental theorem of Nevanlinna, Lemma \ref{llylemma} and $(1)$ part of Lemma \ref{llyimplem}.
\item Applying first fundamental theorem of Nevanlinna, we have
 \begin{align*}
T\left(r,\frac{1}{f(z+c)-f(z)}\right)&\leq T(r,f(z+c)-f(z))+O(1)\\
&\leq T\left(r,\frac{f(z+c)-f(z)}{f(z)}\right)+T(r,f(z))+O(1)\\
&\leq T\left(r,\frac{f(z+c)}{f(z)}\right)+T(r,f(z))+O(1)\\
&=m\left(r,\frac{f(z+c)}{f(z)}\right)+N\left(r,\frac{f(z+c)}{f(z)}\right)+T(r,f(z))+O(1).
\end{align*}
Now applying Lemma \ref{llylemma} and $(1)$ part of Lemma \ref{implem1}, we have
$$T\left(r,\frac{1}{f(z+c)-f(z)}\right)\leq T(r,f(z))+S(r,f).$$

\item Similarly as we did for previous part, we can prove this one also.
\end{enumerate} 
\end{proof}

Next lemma gives estimate of the characteristic function of an exponential polynomial $f$ and this can also be seen in \cite{whl}.
\begin{lemma}\rm\label{whllem}
Suppose $f$ is an entire function given by
$$f(z)=A_{0}(z)+A_{1}(z)e^{w_{1}z^{s}}+A_{2}(z)e^{w_{2}z^{s}}+...+A_{m}(z)e^{w_{m}z^{s}},$$
where $A_{i}(z);0\leq i\leq m$ denote either exponential polynomial of degree $<s$ or polynomial in $z$, $w_{i};1\leq i\leq m$ denote the constants and $s$ denotes a natural number. Then
$$T(r,f)=C(Co(W_{0}))\frac{r^{s}}{2\pi}+o(r^{s}),$$
Here $C(Co(W_{0}))$ is the perimeter of the convex hull of the set $W_{0}=\{0,\overline{w}_{1},\overline{w}_{2},...,\overline{w}_{m}\}$.
Moreover, 
\begin{enumerate}
\item if $A_{0}(z)\not\equiv 0$, then
$$m(r,\frac{1}{f})=o(r^{s}).$$ 
\item if $A_{0}(z)\equiv 0$, then
$$N(r,\frac{1}{f})=C(Co(W))\frac{r^{s}}{2\pi}+o(r^{s}),$$
where $W=\{\overline{w}_{1},\overline{w}_{2},...,\overline{w}_{m}\}.$
\end{enumerate}

\end{lemma}

                 	\section{\textbf{Proof of theorems}}

\begin{proof}[\underline{Proof of Theorem A}]
Let $F(z)=f^{2}(z)g^{(k)}(z)-a(z)$, then
\begin{align*}
2T(r,f(z))&=T\left(r,\frac{F(z)+a(z)}{g^{(k)}(z)}\right)\\
&\leq T(r,F(z)+a(z))+T\left(r,\frac{1}{g^{(k)}(z)}\right)+O(1)
\end{align*}
In a simple way, we have
\begin{equation}\label{eq}
2T(r,f)\leq T(r,F+a)+T\left(r,\frac{1}{g^{(k)}}\right)+O(1)
\end{equation}
Using second fundamental theorem of Nevanlinna for three small functions, we have
\begin{align*}
T(r,F+a)&\leq N(r,F+a)+N\left(r,\frac{1}{F+a}\right)+N\left(r,\frac{1}{F}\right)+S(r,F)\\
&\leq N(r,f^{2}g^{(k)})+N\left(r,\frac{1}{f^{2}g^{(k)}}\right)+N\left(r,\frac{1}{F}\right)+S(r,F)\\
&\leq 2\left(N(r,f)+N(r,\frac{1}{f})\right)+(N(r,g^{(k)})+N\left(r,\frac{1}{g^{(k)}}\right)+N\left(r,\frac{1}{F}\right)+S(r,F)
\end{align*}
Applying Lemma \ref{implem1}, we have
\begin{equation}\label{eq1}
T(r,F+a)\leq N\left(r,\frac{1}{F}\right)+S(r,f)+S(r,g).
\end{equation}
Next using Lemma \ref{implem2} on $g(z)$, we have
\begin{equation}\label{eq2}
T\left(r,\frac{1}{g^{(k)}}\right)\leq T(r,g)+S(r,g).
\end{equation}
From equations \eqref{eq}, \eqref{eq1} and \eqref{eq2}, we have
\begin{equation}\label{imppeq1}
2T(r,f)-T(r,g)\leq N\left(r,\frac{1}{F}\right)+S(r,f)+S(r,g) 
\end{equation}
Similarly, let $G(z)=g^{2}(z)f^{(k)}(z)-a(z)$ then proceeding to same technique as we have done above, we get
\begin{equation}\label{imppeq2}
2T(r,g)-T(r,f)\leq N\left(r,\frac{1}{G}\right)+S(r,f)+S(r,g)
\end{equation}
From equation \eqref{imppeq1} and \eqref{imppeq2}, we have
$$T(r,f)+T(r,g)\leq N\left(r,\frac{1}{F}\right)+N\left(r,\frac{1}{G}\right)+S(r,f)+S(r,g).$$
This implies that at least one of $F(z)=f^{2}(z)g^{(k)}(z)-a(z)$ and $G(z)=g^{2}(z)f^{(k)}(z)-a(z)$ have infinitely many zeros.
\end{proof}	

\begin{proof}[\underline{Proof of Theorem B}]
Let $F(z)=f^{2}(z)L(g)-a(z)$, then we have
\begin{equation}\label{eqq2}
	2T(r,f)\leq T(r,F+a)+T\left(r,\frac{1}{L(g)}\right)+O(1)
\end{equation}
Using second fundamental theorem of Nevanlinna for three small functions, we have
\begin{equation*}
	T(r,F+a)\leq N(r,F+a)+N\left(r,\frac{1}{F+a}\right)+N\left(r,\frac{1}{F}
	\right)+S(r,F).
\end{equation*}
	This gives
\begin{equation}\label{eqq2a}
	T(r,F+a)\leq N(r,f^{2}L(g))+N\left(r,\frac{1}{f^{2}L(g)}\right)+N\left(r,\frac{1}{F}\right)+S(r,f)+S(r,L(g)).
\end{equation}
\begin{enumerate}
\item Let $L(g)=g(z+c)$, then applying $(1)$ part of Lemma \ref{llyimplem} to equation \eqref{eqq2a}, we have
\begin{equation}\label{eqq2b}
T(r,F+a)\leq N\left(r,\frac{1}{F}\right)+S(r,f)+S(r,g).
\end{equation}
Applying $(1)$ part of Lemma \ref{implem3} to equation \eqref{eqq2} and together with equation \eqref{eqq2b}, we have
\begin{equation}\label{eqq2c}
		2T(r,f)-T(r,g)\leq N\left(r,\frac{1}{F}\right)+S(r,f)+S(r,g).
\end{equation}
Similarly, let $G(z)=g^{2}(z)f(z+c)-a(z)$ then proceeding on similar lines as we did for $F(z)$ function, we get
 \begin{equation}\label{eqq2d}
 2T(r,g)-T(r,f)\leq N\left(r,\frac{1}{G}\right)+S(r,f)+S(r,g).
 \end{equation}
From equations \eqref{eqq2c} and \eqref{eqq2d}, we have
$$T(r,f)+T(r,g)\leq N\left(r,\frac{1}{F}\right)+N\left(r,\frac{1}{G}\right)+S(r,f)+S(r,g).$$ 
This implies that at least one of $F(z)=f^{2}(z)g(z+c)-a(z)$ and $G(z)=g^{2}(z)f(z+c)-a(z)$ have infinitely many zeros.

\item Let $L(g)=g(z+c)-g(z)$, then applying $(2)$ part of Lemma \ref{llyimplem} to equation \eqref{eqq2a}, we get equation \eqref{eqq2b}. Also applying $(2)$ part of Lemma \ref{implem3} to equation \eqref{eqq2} and together with equation \eqref{eqq2b}, we get equation \eqref{eqq2c}. Moreover, let $G(z)=g^{2}(z)(f(z+c)-f(z))-a(z)$, then with the same idea as for $F(z)$, we get equation \eqref{eqq2d}. Hence we obtain the required conclusion. 

\item Let $L(g)=g^{(k)}(z+c)$, then applying $(3)$ part of Lemma \ref{llyimplem}  and \ref{implem3} into the equations \eqref{eqq2}, \eqref{eqq2a} and \eqref{eqq2b}, we get equation \eqref{eqq2c}. Next, let $G(z)=g^{2}(z)f^{(k)}(z+c)-a(z)$, then with the same idea as for $F(z)$, we get equation \eqref{eqq2d}. Hence we obtain the required conclusion. 
\end{enumerate}
\end{proof}

\begin{proof}[\underline{Proof of Theorem C}]
Suppose that
\begin{equation}\label{eq3a}
f(z)=\alpha_{1}(z)e^{P_{1}(z)}-k_{1} ~ \text{and} ~
 g(z)=\alpha_{2}(z)e^{P_{2}(z)}-k_{2},
\end{equation}
 where $\alpha_{1}$, $\alpha_{2}$ are non-zero polynomials, $k_{1},k_{2}\in\mathbb{C}\setminus\{0\}$ and $P_{1}, P_{2}$ are non-constant polynomials say $P_{1}(z)=a_{m}z^{m}+a_{m-1}z^{m-1}+...+a_{0}$ and $P_{2}(z)=b_{n}z^{n}+b_{n-1}z^{n-1}+...+b_{0}$.\\
 Differentiating equation \eqref{eq3a} gives
 \begin{equation}\label{eq3b}
 f^{'}(z)=(\alpha_{1}^{'}+\alpha_{1}P_{1}^{'})e^{P_{1}(z)} ~ \text{and} ~ g^{'}(z)=(\alpha_{2}^{'}+\alpha_{2}P_{2}^{'})e^{P_{2}(z)}.
 \end{equation}	
From equation \eqref{eq3a} and \eqref{eq3b}, we have
\begin{align}\label{eq3c}
fg^{'}-a(z)&=(\alpha_{1}e^{P_{1}}-k_{1})(\alpha_{2}^{'}+\alpha_{2}P_{2}^{'})e^{P_{2}}-a(z) \nonumber	\\
&=(\alpha_{1}\alpha_{2}^{'}+\alpha_{1}\alpha_{2}P_{2}^{'})e^{P_{1}+P_{2}}-k_{1}(\alpha_{2}^{'}+\alpha_{2}P_{2}^{'})e^{P_{2}}-a(z) 
\end{align}
and
\begin{align}\label{eq3d}
gf^{'}-a(z)&=(\alpha_{2}e^{P_{2}}-k_{2})(\alpha_{1}^{'}+\alpha_{1}P_{1}^{'})e^{P_{1}}-a(z) \nonumber	\\
&=(\alpha_{2}\alpha_{1}^{'}+\alpha_{1}\alpha_{2}P_{1}^{'})e^{P_{1}+P_{2}}-k_{2}(\alpha_{1}^{'}+\alpha_{1}P_{1}^{'})e^{P_{1}}-a(z).
\end{align}
Now we discuss the following three cases:
\begin{enumerate}
\item If $\deg(P_{1})>\deg(P_{2})$, then equations \eqref{eq3c} and \eqref{eq3d} can be written as
\begin{equation}\label{eq3e}
fg^{'}-a(z)=H_{1}(z)e^{a_{m}z^{m}}+H_{0}(z)
\end{equation}
 and 
\begin{equation}\label{eq3f}
gf^{'}-a(z)=J_{1}(z)e^{a_{m}z^{m}}-a(z),
\end{equation}
where
\begin{align*}
 H_{1}(z)=&(\alpha_{1}\alpha_{2}^{'}+\alpha_{1}\alpha_{2}P_{2}^{'})e^{P_{2}+a_{m-1}z^{m-1}+...+a_{0}},\\ H_{0}(z)=&-(k_{1}\alpha_{2}^{'}+k_{1}\alpha_{2}P_{2}^{'})e^{P_{2}}-a(z),\\
J_{1}(z)=&[(\alpha_{2}\alpha_{1}^{'}+\alpha_{1}\alpha_{2}P_{1}^{'})e^{P_{2}}-(k_{2}\alpha_{1}^{'}+k_{2}\alpha_{1}P_{1}^{'})]e^{a_{m-1}z^{m-1}+...+a_{0}}.
\end{align*}
Applying Lemma \ref{whllem} and first main theorem of Nevanlinna into equations \eqref{eq3e} and \eqref{eq3f} give
$$N\left(r,\frac{1}{fg^{'}-a(z)}\right)=T(r,fg^{'}-a(z))+O(1)$$
and $$N\left(r,\frac{1}{gf^{'}-a(z)}\right)=T\left(r,gf^{'}-a(z)\right)+O(1).$$
Hence $\lambda(fg^{'}-a(z))=\rho(fg^{'}-a(z))$ and $\lambda(gf^{'}-a(z))=\rho(gf^{'}-a(z))$.
This implies that both $fg'-a(z)$ and $gf'-a(z)$ have infinite number of zeros.

\item If $\deg(P_{2})>\deg(P_{1})$, then equations \eqref{eq3c} and \eqref{eq3d} can be written as
\begin{equation}\label{eq3g}
	fg^{'}-a(z)=H_{3}(z)e^{b_{n}z^{n}}-a(z)
\end{equation}
and 
\begin{equation}\label{eq3h}
	gf^{'}-a(z)=J_{3}(z)e^{b_{n}z^{n}}+J_{0}(z),
\end{equation}
where
\begin{align*}
H_{3}(z)=&[(\alpha_{1}\alpha_{2}^{'}+\alpha_{1}\alpha_{2}P_{2}^{'})e^{P_{1}}-(k_{1}\alpha_{2}^{'}+k_{1}\alpha_{2}P_{2}^{'})]e^{b_{n-1}z^{n-1}+...+b_{0}},\\
J_{3}(z)=&(\alpha_{2}\alpha_{1}^{'}+\alpha_{1}\alpha_{2}P_{1}^{'})e^{P_{1}+b_{n-1}z^{n-1}+...+b_{0}},\\
J_{0}(z)=&-(k_{2}\alpha_{1}^{'}+k_{2}\alpha_{1}P_{1}^{'})e^{P_{1}}-a(z).
\end{align*}
Proceeding on similar manner as we deal with $(1)$ part, we obtain
that both $fg'-a(z)$ and $gf'-a(z)$ have infinite number of zeros from the equations  \eqref{eq3g} and \eqref{eq3h}.
\item If  $\deg(P_{1})=\deg(P_{2})$, then equations \eqref{eq3c} and \eqref{eq3d} become
\begin{equation}\label{eq3i}
fg^{'}-a(z)=H_{4}(z)e^{(a_{n}+b_{n})z^{n}}+H_{5}(z)e^{b_{n}z^{n}}-a(z)
\end{equation}
and
\begin{equation}\label{eq3j}
gf^{'}-a(z)=J_{4}(z)e^{(a_{n}+b_{n})z^{n}}+J_{5}(z)e^{a_{n}z^{n}}-a(z),
\end{equation}
where
\begin{align*}
H_{4}(z)=&(\alpha_{1}\alpha_{2}^{'}+\alpha_{1}\alpha_{2}P_{2}^{'})e^{(a_{n-1}+b_{n-1})z^{n-1}+...+(a_{0}+b_{0})},\\
H_{5}(z)=&-(k_{1}\alpha_{2}^{'}+k_{1}\alpha_{2}P_{2}^{'})e^{b_{n-1}z^{n-1}+...+b_{0}},\\
J_{4}(z)=&(\alpha_{2}\alpha_{1}^{'}+\alpha_{1}\alpha_{2}P_{1}^{'})e^{(a_{n-1}+b_{n-1})z^{n-1}+...+(a_{0}+b_{0})},\\
J_{5}(z)=&-(k_{2}\alpha_{1}^{'}+k_{2}\alpha_{1}P_{1}^{'})e^{a_{n-1}z^{n-1}+...+a_{0}}.
\end{align*}
Next we study the following subcases:
\begin{enumerate}[(i)]
\item If $a_{n}\neq \pm b_{n}$, then applying same reason to equations  \eqref{eq3i} and \eqref{eq3j} as we done for $(1)$ part, we obtain that both $fg'-a(z)$ and $gf'-a(z)$ have infinite number of zeros.
\item If $a_{n}=b_{n}$, then  proceeding the same logic as we applied in $(1)$ part, we get same conclusion from equations \eqref{eq3i} and \eqref{eq3j}.
\item If $a_{n}=-b_{n}$, then equations \eqref{eq3i} and \eqref{eq3j} become
\begin{equation}\label{eq3k}
fg^{'}-a(z)=H_{5}(z)e^{b_{n}z^{n}}+H_{4}(z)-a(z)
\end{equation}
and 
\begin{equation}\label{eq3l}
gf^{'}-a(z)=J_{5}(z)e^{-b_{n}z^{n}}+J_{4}(z)-a(z).
\end{equation}
\begin{enumerate}[(I)]
\item If $H_{4}(z)\not\equiv a(z)$ and $J_{4}(z)\not\equiv a(z)$, then applying same reason to equations \eqref{eq3k} and \eqref{eq3l} as we done for $(1)$ part, we obtain that both $fg'-a(z)$ and $gf'-a(z)$ have infinite number of zeros.
\item If $H_{4}(z)\equiv a(z)$, then $J_{4}(z)\not\equiv a(z)$. Otherwise $H_{4}(z)\equiv J_{4}(z)$, this gives $\alpha_{2}(z)-\alpha_{1}(z)=e^{P_{1}-P_{2}+C}$, which is not possible. Here $C$ is an arbitrary constant. By the simple observation, we get  $fg'-a(z)$ has finitely many zeros from equation \eqref{eq3k}. Also proceeding with the same logic as we applied in $(1)$ part, we obtain  $gf^{'}-a(z)$ has infinitely many zeros from the equation \eqref{eq3l}.
\item If $J_{4}(z)\equiv a(z)$, then proceeding with the same logic as we applied in (II) part, we get $fg'-a(z)$ has infinitely many zeros and $gf^{'}-a(z)$ has finitely many zeros.
\end{enumerate}
\end{enumerate}
\end{enumerate}
\end{proof}

\end{document}